\documentclass[10pt]{amsart}
\usepackage{amsfonts,amssymb,amscd,amsmath,enumerate,verbatim,calc,latexsym,pstcol,pst-plot,pst-3d,}
\usepackage[noxcolor]{pstricks}
\input xy
\xyoption{all}

\def\NZQ{\mathbb}               % the font for N,Z,Q,R,C
\def\NN{{\NZQ N}}

\def\RR{{\NZQ R}}

\def\frk{\mathfrak}               % font for "Fraktur"

\def\mm{{\frk m}}

\def\Phi{{\frk N}}
\def\ab{{\bold a}}

\def\opn#1#2{\def#1{\operatorname{#2}}} % to make operators
\opn\chara{char} \opn\length{\ell} \opn\pd{pd} \opn\rk{rk}
\opn\projdim{proj\,dim} \opn\injdim{inj\,dim} \opn\rank{rank}
\opn\depth{depth} \opn\grade{grade} \opn\height{height}
\opn\embdim{emb\,dim} \opn\codim{codim}

\opn\Tr{Tr} \opn\bigrank{big\,rank}
\opn\superheight{superheight}\opn\lcm{lcm}
\opn\trdeg{tr\,deg}%\emph{
\opn\reg{reg} \opn\lreg{lreg} \opn\ini{in} \opn\lpd{lpd}
\opn\size{size}\opn{\mult}{mult}
\opn\div{div} \opn\Div{Div} \opn\cl{cl} \opn\Cl{Cl}
\opn\Spec{Spec} \opn\Supp{Supp} \opn\supp{supp} \opn\Sing{Sing}
\opn\Ass{Ass} \opn\Min{Min}
\opn\Ann{Ann} \opn\Rad{Rad} \opn\Soc{Soc}
\opn\Syz{Syz} \opn\Im{Im} \opn\Ker{Ker} \opn\Coker{Coker}
\opn\Am{Am} \opn\Hom{Hom} \opn\Tor{Tor} \opn\Ext{Ext}
\opn\End{End} \opn\Aut{Aut} \opn\id{id}

\opn\nat{nat}
\opn\pff{pf}%   \pf exists already
\opn\Pf{Pf} \opn\GL{GL} \opn\SL{SL} \opn\mod{mod} \opn\ord{ord}
\opn\Gin{Gin}
\opn\Hilb{Hilb}\opn\adeg{adeg}\opn\std{std}\opn\ip{infpt}
\opn\Pol{Pol}
\opn\sat{sat}
\opn\Var{Var}

\opn\aff{aff} \opn\con{conv} \opn\relint{relint} \opn\st{st}
\opn\lk{lk} \opn\cn{cn} \opn\core{core} \opn\vol{vol}
\opn\link{link} \opn\star{star}
\opn\gr{gr}

\def\pot#1#2{#1[\kern-0.28ex[#2]\kern-0.28ex]}

\opn\dirlim{\underrightarrow{\lim}}
\opn\inivlim{\underleftarrow{\lim}}
\let\union=\cup
\let\sect=\cap

\let\Sect=\bigcap
\let\Dirsum=\bigoplus

\let\to=\rightarrow

\def\Implies{\ifmmode\Longrightarrow \else
        \unskip${}\Longrightarrow{}$\ignorespaces\fi}
\def\implies{\ifmmode\Rightarrow \else
        \unskip${}\Rightarrow{}$\ignorespaces\fi}
\def\iff{\ifmmode\Longleftrightarrow \else
        \unskip${}\Longleftrightarrow{}$\ignorespaces\fi}

\let\:=\colon
\newtheorem{Theorem}{Theorem}[section]
\newtheorem{Lemma}[Theorem]{Lemma}

\newtheorem{Proposition}[Theorem]{Proposition}

\newtheorem{Example}[Theorem]{Example}

\let\epsilon\varepsilon
\let\phi=\varphi
\let\kappa=\varkappa
\def\qed{\ifhmode\textqed\fi
      \ifmmode\ifinner\quad\qedsymbol\else\dispqed\fi\fi}
\def\textqed{\unskip\nobreak\penalty50
       \hskip2em\hbox{}\nobreak\hfil\qedsymbol
       \parfillskip=0pt \finalhyphendemerits=0}
\def\dispqed{\rlap{\qquad\qedsymbol}}

\opn\dis{dis}
\def\pnt{{\raise0.5mm\hbox{\large\bf.}}}

\opn\Lex{Lex}

\begin{document}

\title{Symbolic powers of monomial ideals which are generically complete intersections}

\author{Adnan Aslam}
\thanks{}
\subjclass{13C14, 13D02, 13D25, 13P10}
\address{Adnan Aslam, Abdus Salam School of Mathematical Sciences (ASSMS), GC University, Lahore, Pakistan.
}
\email{adnanaslam15@yahoo.com}

\begin{abstract}
We classify all unmixed monomial ideals $I$ of codimension 2 which are generically a complete intersection and
which have the property that the symbolic power algebra $A(I)$ is standard graded. We give a lower bound for the highest degree of a generator of $A(I)$ in the case that $I$ is a modification of  the vertex cover ideal of a bipartite graph, and show that this highest degree can be any given number.  We furthermore give an upper bound for the highest degree of a generator of the integral closure of  $A(I)$ in the case that $I$ is a monomial ideal which is  generically a complete intersection.

%\vskip 0.4 true cm
% \noindent
%  {\it Key words } : Monomial Ideals, Taylor Complexes, Linear Resolutions, Chordal Graphs.
\end{abstract}

\maketitle

\section*{Introduction}
The aim of the paper is to study the symbolic power algebra of monomial ideals. Classically one defines the symbolic power algebra of prime ideal $P$ in a Noetherian commutative ring $R$ as the graded algebra $\Dirsum_{k\geq 0} P^{(k)}t^k$ where $P^{(k)}$ is the $k$th symbolic power of $P$. In general, this algebra is not finitely generated, even though $R$ is Noetherian. Examples of this bad behavior have been given by Cowsik \cite{C}, Roberts \cite{R} and others. On the other hand, it has been shown in \cite{HHT1} the symbolic power algebra of a monomial ideal is always finitely generated. In the first section of this paper we recall the definition of symbolic powers for monomial ideals. In the case, that the monomial ideal $I$  is squarefree, it well understood when its symbolic power algebra $A(I)$ is standard graded, in other words, when the ordinary powers of $I$ coincide with the symbolic powers  of $I$, see \cite{HHT1},\cite{GRV1}, and \cite{SVV}.

The interest in symbolic powers of monomial ideals is partly due to the fact that symbolic powers of squarefree monomial ideals may be interpreted as vertex cover ideals of graphs and hypergraphs. This point of view has been stressed in the papers \cite{HHT1},\cite{HHT2} and  \cite{HHTZ}. On the other hand, not much is known about the generation of symbolic power algebras in the case that the monomial ideal is not squarefree. This paper is an attempt to  better understand symbolic power algebras of non-squarefree monomial ideals. The situation in this more general case is much more complicated. While the symbolic power algebra of an unmixed  squarefree monomial ideal of codimension 2 is generated in degree at most 2, the highest degree $d(A(I))$ of a generator of the algebra $A(I)$, where $I$ is an unmixed  monomial ideal of codimension 2 can be  any given number, as is shown in a family of examples in Section 2. In Section 1 we introduced some concepts and first recall a few facts about vertex cover algebras which will be needed in the following sections. The main result of Section 2 is Theorem~\ref{main} in which it is shown that the symbolic power algebra of monomial ideal $I$ of codimension 2 which is generically a complete intersection  is standard graded if and only if $I$ is a trivial modification of a vertex cover ideal of a bipartite graph, which means that it is obtained from the vertex cover ideal simply by variable substitutions. In Section 3 we give a general lower bound for the number $d(A(I))$ in the case that $I$ is a modification of a vertex cover ideal of a bipartite graph. To give an upper bound for $d(A(I))$ is much harder. However, in Section 4,  an explicit upper bound for the highest degree of a generator of the integral closure $\overline{A(I)}$ of $A(I)$ is given when $I$ is a monomial ideal which is an arbitrary generically complete intersection with no condition on the codimension. It remains an open question how the numbers $d(A(I))$ and $d(\overline{A(I)})$ are related to each other.

\section{A review on symbolic powers of monomial ideals}

Let $K$ be a field and $S=K[x_1,\ldots,x_n]$ the polynomial ring in $n$ indeterminates  over $K$, and let $I\subset S$ be a monomial ideal. Then $I$ has a unique irredundant presentation
\[
I=\Sect_{i=1}^m C_i.
\]
as an intersection of irreducible monomial ideals, see \cite{HH1}. The irreducible monomial ideals $C_i$ are all of the form $(x_{i_1}^{a_1},\ldots, x_{i_k}^{a_k})$. In particular, they are monomial complete intersections. One obtains from this presentation a canonical presentation of $I$ as an intersection of primary ideals
\begin{eqnarray}
\label{primary}
I=\Sect_{i=1}^rQ_i.
\end{eqnarray}
where each $Q_i$ is $P_i$-primary and is defined to be the intersection of all $C_j$ with $\sqrt{C_j}=P_i$. For example, if $I=(x_1^3,x_2^3,x_1^2x_3^2,x_1x_2x_3^2,x_2^2x_3^2)$ then $I$ has the irredundant presentation as intersection of irreducible ideals
\[
I=(x_1^3,x_2^3,x_3^2)\cap(x_1^2,x_2)\cap(x_1,x_2^2).
\]
We have $\Ass(x_1^2,x_2)=(x_1,x_2)=\Ass(x_1,x_2^2)$. Intersecting $(x_1^2,x_2)$ and $(x_1,x_2^2)$ we obtain $(x_1,x_2)$-primary ideal $(x_1^2,x_1x_2,x_2)$ and hence the irredundant primary decomposition
\[
I=(x_1^3,x_2^3,x_3^2)\cap(x_1^2,x_1x_2,x_2^2).
\]

Given a primary decomposition as in (\ref{primary}) one defines the {\em $k$th symbolic power} of $I$ to be
\[
I^{(k)}=\Sect_{i=1}^rQ_i^k.
\]
Obviously one has $I^k\subset I^{(k)}$ for all $k$. Note that there is no ambiguity in the definition of $I^{(k)}$ since the presentation of $I$ as an intersection of primary ideals as given in (\ref{primary}) is uniquely determined by $I$. Furthermore, if $I$ is generically complete intersection, as will be assumed in the following sections, this definition of $I^{(k)}$ coincide with one given in \cite{HHT1}

One of the basic questions related to symbolic powers is to find conditions when   $I^k= I^{(k)}$ for all $k$. Since $I^{(k)}I^{(l)}\subset I^{(k+l)}$ for all $k$ and $l$,  one may consider the so-called {\em symbolic power algebra}
\[
A(I)=\Dirsum_{k\geq 0}I^{(k)}t^k\subset S[t].
\]
This algebra contains the  Rees algebra $R(I)=\Dirsum_{k\geq 0}I^kt^k$, and the  following statements are equivalent:
\begin{enumerate}
\item[(a)] $I^k=I^{(k)}$ for all $k$.
\item[(b)] $R(I)=A(I)$.
\item[(c)] $A(I)$ is a standard graded $S$-algebra, i.e., $A(I)$ is generated over $S$ by elements of degree $1$.
\end{enumerate}
It is known that $A(I)$ is always finitely generated, see \cite[Theorem 3.2]{HHT1}. For squarefree monomial ideals this has first been shown in \cite[proposition1]{L}. In case that $I$ is a squarefree monomial ideal it is well understood when $A(I)$ is standard graded, see \cite[Theorem 5.1]{HHT1}. In particular, the following results will be of importance for our considerations. Assume that $I$ is a squarefree monomial ideal  which is obtained as an intersection of ideals of the form $(x_i,x_j)$. To such an ideal we can associate a graph $G$ on the vertex set $[n]$ for which $\{i,j\}$ is an edge of $G$ if and only if $(x_i,x_j)$ appears in the intersection of $I$. Thus we can write
\[
I=\Sect_{\{i,j\}\in E(G)}(x_i,x_j).
\]
The monomial generators of $I$ correspond to the vertex covers of $G$. Indeed, $u=x_{i_1}x_{i_2}\cdots x_{i_k}$ is a generator of $I$ if and only if the set $C=\{i_1,i_2,\ldots,i_k\}$ is a vertex cover of $G$, that is, $C\sect \{i,j\}\neq \emptyset $ for all edges $\{i,j\}\in E(G)$. One therefore calls $I$ the {\em vertex cover ideal} of $G$ and denotes it by $I_G$. The vertex cover ideal is in fact the Alexander dual of the edge ideal associated to the graph $G$.  In the next section we will use the following result that can be found in  \cite[Theorem 5.1(b)]{HHT1},
\begin{Theorem}
\label{bipartite}
The following conditions are equivalent:
\begin{enumerate}
\item[{\em (a)}] $R(I_G)= A(I_G)$.
\item[{\em (b)}] $I_G^2=I_G^{(2)}$.
\item[{\em (c)}] $G$ is a bipartite graph.
\end{enumerate}
\end{Theorem}

\section{Unmixed monomial ideals of codimension 2 which are generically a complete intersection}

A monomial ideal $I\subset S= K[x_1,\ldots,x_n]$ of codimension 2 is unmixed and generically a complete intersection if and only if
\[
I=\Sect_{1\leq i<j\leq n}(x_i^{a_{ij}}, x_j^{a_{ji}}),
\]
where  $A=(a_{ij})$ is an $n\times n$-matrix of non-negative integers.  Whenever $a_{ij}=0$ or $a_{ji}=0$, we have $(x_i^{a_{ij}}, x_j^{a_{ji}})=S$, and we may omit this ideal in the intersection. Thus we may rewrite the above intersection as
\[
I=\Sect_{\{i,j\}\in E(G)}(x_i^{a_{ij}}, x_j^{a_{ji}}),
\]
where $G$ is a simple graph on the vertex set $[n]$ and where for each edge $\{i,j\}$ of $G$  the numbers $a_{ij}$ and $a_{ji}$ are positive integers. The graph $G$ introduced is uniquely determined by $I$, and  $\sqrt{I}= \Sect_{\{i,j\}\in E(G)}(x_i, x_j)$ is the vertex cover ideal $I_G$ of $G$.

We say that $J=\Sect_{1\leq i<j\leq n}(x_i^{b_{ij}}, x_j^{b_{ji}})$ is a {\em trivial modification} of
$I=\Sect_{1\leq i<j\leq n}(x_i^{a_{ij}}, x_j^{a_{ji}})$, if there exist integers  $c_i>0$ such that $b_{ij}=c_ia_{ij}$ for all $i$ and $j$.
In particular,  $I$ is a trivial modification  of $\Sect_{\{i,j\}\in E(G)}(x_i, x_j)$, if $a_{kl}$ does not depend on $l$ (and only on $k$).

We begin with an example which shows that for a  non-squarefree monomial ideal $I$, the algebra $A(I)$ may have arbitrarily high degree generators.

\begin{Example}
\label{goodexample}
{\em Given an integer $n$. Let $R$  be  the polynomial ring $\mathbb Q[a,b,c]$ and $I=(a^n,ab,bc)$. Then the highest degree of an algebra generator of $A(I)$ is precisely $n$. In particular, the highest degree of a generator of the symbolic power of a monomial ideal of height 2 may be arbitrarily large.}
{\em
For the proof we note first that
\[
I=(a^n,b)\cap(a,c).
\]
Hence $I$ is a monomial ideal of codimension 2.
We  claim that the $k$th symbolic power $I^{(k)}$ of $I$  is minimally generated by the monomials
\begin{eqnarray}
\label{generators}
&&a^{i}b^{k-\lfloor{i/n}\rfloor}c^{k-i} \quad \text{with}\quad  0\leq i \leq k,  \\
&&a^{{nj}}b^{k-j}\quad \text{with} \quad \lfloor{k/n}\rfloor <j \leq k .\nonumber
\end{eqnarray}
Indeed, we have that
\[
I^{(k)}= (a^n,b)^k\cap(a,c)^k.
\]
Therefore $I^{(k)}$ is generated by the monomials
\begin{eqnarray}
\label{allgenerators}
a^{\max{(nj,i)}} b^{k-j} c^{k-i},\quad 0\leq i, j\leq k.
\end{eqnarray}
Therefore, we have the generators
\begin{eqnarray}
\label{first}
a^{i}b^{k-j}c^{k-i}\quad  \text{with}\quad 0\leq i,j\leq k \quad\text{and}\quad nj\leq i,
\end{eqnarray}
and
\begin{eqnarray}
\label{second}
a^{nj}b^{k-j}c^{k-i}\quad  \text{with}\quad 0\leq i,j\leq k \quad\text{and}\quad nj> i.
\end{eqnarray}
Among the generators (\ref{first}) it is enough to take the generators
\[
a^{i}b^{k-\lfloor{i/n}\rfloor}c^{k-i}\quad \text{with} \quad 0\leq i \leq k.
\]

And among the generators (\ref{second}) it is enough to take the generators
\[
a^{{nj}}b^{k-j}\quad \text{with} \quad \lfloor{k/n}\rfloor <j \leq k .
\]

This implies the assertion about the generators of $I^{(k)}$.}

{\em
Next we claim that the element  $a^nb^{n-1}\in I^{(n)}$ is a minimal generator of the algebra. Suppose this is not the case. Then, since $a^nb^{n-1}$ is a minimal generator of $I^{(n)}$,  there exist integers $k_1,k_2> 0$ with $k_1+k_2=n$,  a minimal generators  $u$ of $ I^{(k_1)}$ and a minimal generator $v$ of  $I^{(k_2)}$ such that  $a^nb^{n-1}=uv$.  Since $a^nb^{n-1}$ does not contain the factor $c$, it follows from (\ref{generators}) that $u=a^{k_1}b^{k_1-\lfloor k_1/n\rfloor}$ and  $v=a^{k_2}b^{k_2-\lfloor k_2/n\rfloor}$ with $k_i<n$, $k_1+k_2=n$ and $k_1+k_2-\lfloor k_1/n\rfloor-\lfloor k_2/n\rfloor=n-1$. Since $\lfloor k_i/n\rfloor=0$ for $i=1,2$, this is impossible.}

{\em
It remains to be shown that for $k>n$, the monomials in (\ref{generators}) can be presented  as products of generators of the symbolic powers $I^{(l)}$ with $0<l<k$. In fact,
\begin{itemize}
\item for $0\leq i<k$ we have
$
a^{i}b^{k-{\lfloor{i/n}\rfloor}}c^{k-i}= (a^{i}b^{k-1-{\lfloor{i/n}\rfloor}}c^{k-1-i})(bc)
$
with\\  $a^{i}b^{k-1-{\lfloor{i/n}\rfloor}}c^{k-1-i}\in I^{(k-1)}$ and $bc \in I$
\item
$
a^{k}b^{k-{\lfloor{k/n}\rfloor}}=(a^{k-1}b^{k-1-{\lfloor{k/n}\rfloor}})(ab)
$
with $a^{k-1}b^{k-1-{\lfloor{k/n}\rfloor}}\in I^{(k-1)}$ and $ab\in I$.
\item for $\lfloor k/n\rfloor <j<k$ we have
$
a^{nj}b^{k-j}=(a^{n(j-1)}b^{k-1-j})(ab)
$
with \\ $a^{n(j-1)}b^{k-1-j}\in I^{(k-1)}$ and $ab\in I$.
\item
$
a^{nk}=a^{n(k-1)}a^n
$
with $a^{n(k-1)}\in I^{k-1}$ and $a^n\in I$.
\end{itemize}}
\end{Example}

Next  we classify all unmixed monomial ideals $I$ of codimension 2 which are generically a complete intersection and
which have the property that $A(I)$ is standard graded.

Let $A$ be finitely generated graded $S$-algebra. We denote by $d(A)$ the highest degree of the minimal generators of $A$. By using this notation we have that $A(I)$ is standard graded if and only if $d(A(I))=1$.

Before proving the main result of this section we observe

\begin{Proposition}
\label{trivialmod}
Let $I$ be an arbitrary monomial ideal and $J$ a trivial modification  of $I$. Then $d(A(J))=d(A(I))$.
\end{Proposition}

\begin{proof}
Let $c_1,\ldots, c_n$ be positive integers such that $J$ arises from $I$ by the substitution $x_i\mapsto x_i^{c_i}$ for $i=1,\ldots,n$, and let  $\varphi\: S\to S$ be the $K$-algebra homomorphism with $\varphi(x_i)=x_i^{c_i}$. For an ideal $J=(f_1,\ldots,f_m)\subset S$ we denote by $\varphi(J)$ the ideal in $S$ generated by  the elements $\varphi(f_i)$, $i=1,\ldots,m$. Then we have $\varphi(I^k)=\varphi(I)^k$ for all $k$. We claim that for any monomial ideal $I$ we also have $\varphi(I^{(k)})=\varphi(I)^{(k)}$ for all $k$. To show this it suffices to prove that whenever $I$ and $J$ are monomial ideals, then $\varphi(I\sect J)=\varphi(I)\sect \varphi(J)$. Indeed, if $I=(u_1,\ldots,u_r)$ and $J=(v_1,\ldots,v_s)$,  then $I\sect J=(\{\lcm(u_i,v_j)\}_{i=1,\ldots,r\atop j=1,\ldots,s})$. Thus the claim follows from the fact that
\[
\varphi(\lcm(u,v))=\lcm(\varphi(u),\varphi(v))\quad \text{for all monomials}\quad u,v.
\]
Let $t$ be an indeterminate over $S$. Then $\varphi$ induces a graded $K$-algebra homomorphism $\hat{\varphi}\: S[t]\to S[t]$ with $\hat{\varphi}(\sum_kf_kt^k)=\sum_k\varphi(f_k)t^k$. Restricting $\hat{\varphi}$ to $A(I)=\Dirsum_k I^{(k)}t^k$ we obtain
\[
\hat{\varphi}(A(I))=\Dirsum_k \varphi(I^{(k)})t^k=\Dirsum_k \varphi(I)^{(k)}t^k=\Dirsum_k J^{(k)}t^k=A(J).
\]
This shows that $\hat{\varphi}\: A(I)\to A(J)$ is surjective. Since $\hat{\varphi}$ is injective, it follows that $A(I)$ and $A(J)$ are isomorphic, and in particular, $d(A(I))=d(A(J))$, as desired.
\end{proof}

\begin{Theorem}
\label{main}
Let $I$ be an unmixed monomial ideal of codimension 2 which is generically a complete intersection. Then the following conditions are equivalent:
\begin{enumerate}
\item[{\em (a)}] $I^k=I^{(k)}$ for all $k$.
\item[{\em (b)}] $I^2=I^{(2)}$.
\item[{\em (c)}] $I$ is a trivial modification of a vertex cover ideal of a bipartite graph.
\end{enumerate}
\end{Theorem}

\begin{proof}
The implication (a) \implies (b) is trivial. Next we show that (c)\implies (a): The vertex cover ideal $J$ of a bipartite graph is standard graded, see Theorem~\ref{bipartite}. In other words, $d(A(J))=1$. Since $I$ is a trivial modification of $J$, Proposition~\ref{trivialmod} implies that $d(A(I)))=1$

It remains to be shown that (b) \implies (c). Assume first that $\sqrt I = I_G$ for a bipartite graph $G$ and $I$ is a not trivial modification of vertex cover ideal of  a bipartite graph $G$. It is convenient to identify the edges of $G$ with variables in the polynomial ring. Thus let  the vertex set $V(G)$ of $G$ be $\{x_1,\ldots,  x_r\}\union \{y_1,\,\ldots,y_s\}$ with the edge set $E(G)\subset \{x_1,\ldots, x_r\}\times \{y_1,\ldots, y_s\}$. Then
\[
I=\Sect_{\{x_i,y_j\}\in E(G)}(x_i^{a_{ij}}, y_j^{b_{ij}}).
\]
Since we assume that $I$ is not a trivial modification of $I_G$ there exits an integer $i$, say $i=1$, such that
\[
I_1=\Sect_{\{x_1,y_j\}\in E(G)}(x_1^{a_{1j}}, y_j^{b_{1j}})
\]
is not a trivial modification of $\Sect_{\{x_1,y_j\}\in E(G)}(x_1, y_j)$.
Let $G_1$ be the subgraph of $G$ with $E(G_1)=\{\{x_1,y_j\}\:\; \{x_1,y_j\}\in E(G)\}$. Then
$I=I_1\sect I_2$,  where
\[
I_2=\Sect_{\{x_i,y_j\}\in E(G)\atop i\neq 1}(x_i^{a_{ij}}, y_j^{b_{ij}}),
\]
and $I_1$ is not a trivial modification of $I_{G_1}$.

Let $f=x_2\cdots x_r$. Then for the localization with respect to $f$ we have  $(I_1^2)_f=(I^2)_f$ and $(I_1^{(2)})_f=(I^{(2)})_f$.
Assume we know already  $I_1^2\neq I_1^{(2)}$. If $(I_1^2)_f\neq (I_1^{(2)})_f$, then  $(I^2)_f\neq (I^{(2)})_f$. But then also
$I^2\neq I^{(2)}$, as we want to show. Thus we have reduced the problem to show that for an  ideal of the type
\[
I=(x_1^{a_1},y_1^{b_1})\sect \cdots \sect (x_1^{a_r},y_r^{b_r})
\]
with $a_1\geq a_2\geq \cdots \geq a_r$ and one of the inequalities is strict. Since $I$ is a trivial modification of the  ideal $(x_1^{a_1},y_1)\sect \cdots \sect (x_1^{a_r},y_r)$,  we may as well assume that
all $b_i=1$, see Proposition~\ref{trivialmod}. Take the smallest $i$ such that $a_1>a_i$. Then  for $a_1\geq 2a_i$,
\[
x_1^{a_1}y_1y_2\cdots y_{i-1}\in I^{(2)}
\]
and for $a_1<2a_i$,
\[
x_1^{2a_i}y_1y_2\cdots y_{i-1}\in  I^{(2)}
\]
We have
\[
I =(x_1^{a_1}, x_1^{a_2}y_1, x_1^{a_3}y_1y_2,\ldots, x_1^{a_i}y_1y_2\cdots y_{i-1},\ldots , y_1\cdots y_r).
 \]
It is obvious that $x_1^{a_1}y_1y_2\cdots y_{i-1}\not\in I^2$.  On the other hand, the possible product of two generators of $I$ to get  $x_1^{2a_i}y_1y_2\cdots y_{i-1}$ would be $x_1^{a_1}( x_1^{a_i}y_1y_2\cdots y_{i-1})$. But this is also not possible, because $a_1+a_i>2a_i$.

Assume now that $G$ is not bipartite. Then $G$ contains an odd cycle. If this cycle has a chord,  then this chord decomposes this cycle into two smaller cycles, where one of them is  again an odd cycle. This argument shows that $G$ contains an odd cycle with  no chord. Suppose first that this cycle has length $>3$. By a localization argument as in the first part of the proof (c)\implies (b) we may assume that $G$ coincides with this cycle. Say, the edges of $G$ are  $\{1,n\}$ and $\{i,i+1\}$ for $i=1,\ldots,n-1$ with $n>3$. Localizing $I$ at $f=x_3x_5\cdots x_{n-1}$ we are reduced to show that  for the ideal $J=(x_1^{a_1},x_2^{b_1})\sect(x_1^{a_2}, x_n^{b_2})$  we have that $J^2\neq J^{(2)}$. Without loss of generality we may assume that $a_1> a_2$ and $b_1=1=b_2$ see Proposition~\ref{trivialmod}. If $a_1\geq 2a_2$ then $x_1^{a_1}x_2\in J^{(2)}\setminus J^2$, if $a_1<2a_2$ then $x_1^{2a_2}x_2\in J^{(2)}\setminus J^2$, as desired.

Finally we assume that $E(G)= \{\{1,2\},\{2,3\},\{1,3\}\}$. Then
\[
I=(x_1^{a_1},x_2^{b_1})\sect (x_2^{b_2},x_3^{c_1})\sect (x_3^{c_2},x_1^{a_2}).
\]
We want so show that $I^{(2)}\neq I^2$.  Notice that
\begin{eqnarray*}
I&=&(x_1^{\max(a_1,a_2)}x_2^{b_2},x_1^{a_1}x_2^{b_2}x_3^{c_2},x_1^{\max(a_1,a_2)}x_3^{c_1},x_1^{a_1}x_3^{\max(c_1,c_2)},
x_1^{a_2}x_2^{\max(b_1,b_2)},x_2^{\max(b_1,b_2)}x_3^{c_2},\\
&&x_1^{a_2}x_2^{b_1}x_3^{c_1},x_2^{b_1}x_3^{\max(c_1,c_2)}).
\end{eqnarray*}
Assume that $a_1 \geq a_2$  then
\begin{eqnarray*}
I&=&(x_1^{a_1}x_2^{b_2},x_1^{a_1}x_2^{b_2}x_3^{c_2},x_1^{a_1}x_3^{c_1},x_1^{a_1}x_3^{\max(c_1,c_2)},
x_1^{a_2}x_2^{\max(b_1,b_2)},x_2^{\max(b_1,b_2)}x_3^{c_2},x_1^{a_2}x_2^{b_1}x_3^{c_1},\\
&&x_2^{b_1}x_3^{\max(c_1,c_2)}).
\end{eqnarray*}
Then it follows that
\[
x_1^{a_1}x_2^{\max{(b_1,b_2)}}x_3^{c_1} \in I^{(2)}\setminus I^2,\quad\text{if}\quad a_1 \geq 2a_2,
\]
and
\[
x_1^{a_1}x_2^{\max{(b_1,b_2)}}x_3^{\max{(c_1,c_2)}} \in I^{(2)}\setminus I^2, \quad \text{if}\quad a_2\leq a_1 < 2a_2.
\]
\end{proof}

\section{Bounds for the generators of the algebra $A(I)$}

 In this section we want to give a lower bound for $d(A(I))$ in the case that $I$ is a modification of  the vertex cover ideal of a bipartite graph.
\begin{Theorem}
\label{lowerbound}
Let $I\subset S$ is an ideal, $S=K[x_1,\ldots,x_r,y_1,\ldots,y_s]$ and $a_{ij}$ and $b_{ij}$ are non-negative integers with $1\leq i\leq r$ and $1\leq j\leq s$. Let $G$ be a bipartite graph with vertex set $V(G)=\{x_1,\ldots,  x_r, y_1,\,\ldots,y_s\}$ and edge set $E(G)\subset \{x_1,\ldots,  x_r\}\times \{y_1,\,\ldots,y_s\}$, and let
\[
I=\Sect_{\{x_i,y_j\}\in E(G)}(x_i^{a_{ij}}, y_j^{b_{ij}})
\]
Then
\[
d(A(I))\geq \max\left\{\max_{i,j,l}\Big\{\frac{a_{ij}}{\gcd(a_{ij},a_{il})}\Big\}, \max_{i, j,l}\Big\{\frac{b_{ij}}{\gcd(b_{ij},b_{il})}\Big\}\right\}.
\]
\end{Theorem}

\begin{proof}
Let $f\in S=A_0(I)$ be a squarefree monomial, then $A(I)_f$ is again naturally graded because $f$ is an element of degree $0$ in $A(I)$, and the generators of $A(I)$ are also generators of $A(I)_f$. Hence it follows that $d(A(I)_f)\leq d(A(I))$. Thus any lower bound for $d(A(I)_f)$ is also lower bound for $d(A(I))$.

Notice that
\[(I^{(k)})_f = \Sect_{\{x_i,y_j\}\in E(G)\atop \{x_i,y_j\}\sect \supp(f)=\emptyset}(x_i^{a_{ij}}, y_j^{b_{ij}})^k S_f
\]
Therefore
\[
(I_f)^{(k)}=(I^{(k)})_f.
\]
This implies that $A(I)_f=A(I_f)$. Let
\[
J=\Sect_{{\{x_i,y_j\}\in E(G)}\atop \{x_i,y_j\}\sect \supp(f)=\emptyset}(x_i^{a_{ij}}, y_j^{b_{ij}})\subset K[\{x_i,y_j: x_i,y_j\notin \supp(f)\}]
\]
Then $d(A(I_f))=d(A(J))$,  because $A(I_f)=A(J)\otimes_k K[\{ x_i,x_i^{-1}, y_j,y_j^{-1}\:\; x_i,y_j\in \supp(f)\}]$.

We may assume that
\[
m=\max_{i,j,l}\{a_{ij}/\gcd(a_{ij},a_{il})\}\geq \max_{i,j,l}\{b_{ij}/\gcd(b_{ij},b_{il})\}.
\]
Let $i$ be an integer for which there exist $j$ and $l$ such that $a_{ij}/\gcd(a_{ij},a_{il})$ takes the maximal value $m$. We may assume that $i=1$
and $j=1$. If $m=1$, the assertion of the theorem is trivial. Thus we may assume that $m>1$. This then implies that $l\neq j$ and $a_{ij}>a_{il}$. We may assume that $l=2$. Localizing at $f=x_2\cdots x_ry_3\cdots y_s$ and applying the above considerations,  we may assume that
\[
I=(x_1^{a_{11}}, y_1^{b_{11}})\sect (x_1^{a_{12}}, y_2^{b_{12}}) \quad \text{with}\quad a_{11}>a_{12}.
\]
Now $I$ is a trivial modification of the ideal $(x_1^a, y_1)\sect (x_1^b, y_2)$ with $a>b$, where $a=a_{11}/\gcd(a_{11},a_{12})$ and $b= a_{12}/\gcd(a_{11},a_{12})$.

Thus, according to the Proposition~\ref{trivialmod},  it is enough to be shown that  $d(A(I))\geq a$ for  $I=(x_1^a, y_1)\sect (x_1^b, y_2)$ with $a>b$ and $\gcd(a,b)=1$.

In order to prove this, we first
claim that the $k$th symbolic power $I^{(k)}$ of $I$  is minimally generated by the monomials
\begin{eqnarray}
\label{generators1}
x_1^{ai}y_1^{k-i}y_2^{k-j}\quad \text{with}\quad 0\leq i,j\leq k \quad\text{and}\quad j\leq \lfloor ai/b \rfloor.
\end{eqnarray}
and the monomials
\begin{eqnarray}
\label{generators2}
x_1^{jb}y_1^{k-i}y_2^{k-j}\quad \text{with}\quad 0\leq i,j\leq k \quad\text{and}\quad j> \lfloor ai/b \rfloor.
\end{eqnarray}
Indeed, we have that
\[
I^{(k)}= (x_1^a,y_1)^k\cap(x_1^b,y_2)^k.
\]
Therefore $I^{(k)}$ is generated by the monomials
\begin{eqnarray}
x_1^{\max{(ai,bj)}} y_1^{k-i} y_2^{k-j},\quad 0\leq i, j\leq k.\nonumber
\end{eqnarray}
Therefore, we have the generators
\begin{eqnarray}
\label{first1}
x_1^{ai}y_1^{k-i}y_2^{k-j}, \quad  \text{with}\quad 0\leq i,j\leq k \quad\text{and}\quad ai\geq bj,
\end{eqnarray}
and
\begin{eqnarray}
\label{second2}
x_1^{bj}y_1^{k-i}y_2^{k-j}, \quad  \text{with}\quad 0\leq i,j\leq k \quad\text{and}\quad ai< bj.
\end{eqnarray}
Thus we obtain exactly the monomials listed in (\ref{generators1}) and  (\ref{generators2}).

Next we claim that the element  $x_1^{ab}y_1^{a-b}\in I^{(a)}$ belongs to the minimal set of generators of the algebra. This will then show that $d(A(I))\geq a$. Suppose $x_1^{ab}y_1^{a-b}t^a$ is not a minimal generator of $A(I)$.  Then, since $x_1^{ab}y_1^{a-b}$ is a minimal generator  of $I^{(a)}$, there exist integers $k_1,k_2> 0$ with $k_1+k_2=a$, a minimal generator  $u$ of $ I^{(k_1)}$ and a minimal generator $v$  of $ I^{(k_2)}$ such that  $x_1^{ab}y_1^{a-b}=uv$.  Since $x_1^{ab}y_1^{a-b}$ does not contain the factor $y_2$, it follows from (\ref{generators1}) and (\ref{generators2}) that the monomials $u$ and $v$ are of the form  $x_1^{ai}y_1^{c-i}$ or
$x_1^{bc}y_1^{c-i}$ where  $c=k_1$ or $c=k_2$. We have to distinguish several cases:

\medskip
(i) If $u=x_1^{ai_1}y_1^{k_1-i_1}$ with $ai_1\geq bk_1$ and $v=x_1^{ai_2}y_1^{k_2-i_2}$ with $ai_2\geq bk_2$, then we have
\[
x_1^{ab}y_1^{a-b}=uv=x_1^{a(i_1+ i_2)}y_1^{k_1+k_2-(i_1+ i_2)}=x_1^{a(i_1+i_2)}y_1^{a-(i_1+i_2)}.
\]
This implies that $i_1+i_2=b$. Adding the inequalities $ai_1\geq bk_1$ and $ai_2 \geq bk_2$, and using $k_1+k_2=a$ and $i_1+i_2=b$,  we get
\[
ab=ai_1+ai_2\geq bk_1+bk_2=ba,
\]
which implies that $ai_1= bk_1$ and $ai_2 = bk_2$. Since $\gcd(a,b)=1$, we see that   $i_1=bc$ and $k_1=ac$ for some  integer  $c>0$. This   is not possible, because $k_1<a$.

\medskip
(ii) If $u=x_1^{ai_1}y_1^{k_1-i_1}$ with $ai_1\geq bk_1$ and $v=x_1^{bk_2}y_1^{k_2-i_2}$ with $ai_2< bk_2$, then we have
\[
x_1^{ab}y_1^{a-b}=uv=x_1^{ai_1+bk_2}y_1^{k_1+k_2-(i_1+i_2)}=x_1^{ai_1+bk_2}y_1^{a-(i_1+i_2)}.
\]
This implies that $ai_1+bk_2=ab$ and $i_1+i_2 = b$. Using $ai_2< bk_2$ in the equation $ai_1+bk_2=ab$,  we get $a(i_1+i_2)<ab$ which contradicts the fact that $i_1+i_2 = b$.

\medskip
(iii)  If $u=x_1^{bk_1}y_1^{k_1-i_1}$ with $ai_1< bk_1$ and $v=x_1^{bk_2}y_1^{k_2-i_2}$ with $ai_2< bk_2$, then we have
\[
x_1^{ab}y_1^{a-b}=uv=x_1^{b(k_1+k_2)}y_1^{k_1+k_2-(i_1+i_2)}=x_1^{2bk_2}y_1^{a-(i_1+i_2)}.
\]
This implies that $b(k_1+k_2) = ab$ and $i_1+i_2= b$. Using $ai_1<bk_1$ and $ai_2<bk_2$ in the equation $b(k_1+k_2) = ab$, we get $a(i_1+i_2)<ab$ which contradicts the fact that $i_1+i_2 = b$.
\end{proof}

One would  like to obtain a similar lower bound for ideals which are modifications of the vertex cover ideals of an arbitrary graph. The first non-trivial case to be considered would be
\[
I=(x^a,y^b)\sect (y^c,z^d)\sect (x^e,z^f)\quad \text{with} \quad \gcd(a,e)=1,\;   \gcd(b,c)=1\quad \text{and} \quad  \gcd(d,f)=1.
\]
We conjecture that $d(A(I))\geq \max\{a,b,c,d,e,f\}$. However, in general this lower bound, if correct,  is not strict as the following examples shows: let
$I=(a^2,b)\sect(b,c)\sect (a,c)$. Then $a^2b^2ct^3\in I^{(3)}t^3$ is a minimal generator of $A(I)$.

\section{The integral closure of $A(I)$ for a monomial ideal $I$ which is generically a complete intersection}

It seems to be pretty hard to find a general  upper bound for the number $d(A(I))$. However for
the integral closure $\overline{A(I)}$ of $A(I)$ where $I$ is a monomial ideal  which is generically a complete intersection such bound can be given, with respect to the number of indeterminates and degrees of the powers of the variables which appears in the irredundant irreducible decomposition of $I$. Since $I$ is generically a complete intersection it is of the form
\[
I=\Sect_{l=1}^m \mm^{\ab_l},
\]
where for a nonzero vector $\ab=(a_1,\ldots,a_n)\in\NN^n$ we set $\mm^{\ab}=(\{x_i^{a_i}\}_{i \in \supp(\ab)})$.

\begin{Lemma}
\label{belongsto}
The integral closure of $A(I)$ is spanned as a $K$-vector space by all monomials $x_1^{c_1}\cdots x_n^{c_n}t^k$ for which the exponent vector $(c_1,\ldots,c_n,k)$ satisfies the following inequalities:
\[
c_1\geq 0,\ldots,c_n\geq 0,\,  k\geq 0\quad{and}\quad \sum_{i\in \supp(\ab_l)}\frac{c_i}{a_{li}}-k\geq 0, \quad 1\leq l\leq m.
\]
\end{Lemma}

\begin{proof}. The integral closure $\overline{A(I)}$ of $A(I)$ is generated as a $K$-vector space by all monomials $u=x_1^{c_1}\cdots x_n^{c_n}t^k\in S[t]$ such that $u^r\in I^{(rk)}t^{rk}$. Now $u^r\in I^{(rk)}t^{rk}$ if and only if $x_1^{rc_1}\cdots x_n^{rc_n}\in(x_1^{a_{l1}},\ldots,x_n^{a_{ln}})^{rk}$ for all $l=1,\ldots,m$. Note that $x_1^{rc_1}\cdots x_n^{rc_n}\in (x_1^{a_{l1}},\ldots,x_n^{a_{ln}})^{rk}$ if and only if there exists non-negative integers $j_1,\ldots,j_n$ with $j_1+\cdots+j_n=rk$ and such that $rc_1\geq j_1a_{l1},\ldots, rc_n\geq j_na_{ln}$. These inequalities are satisfied for $c_1,\ldots,c_n$ if and only if $\sum_{i\in \supp(\ab_l)}c_i/a_{li}\geq k$. This yields the desired conclusion.
\end{proof}

\begin{Theorem}
\label{upper bound}
Let $I\subset S=K[x_1,\ldots,x_n]$ be a monomial ideal which is generically a complete intersection, and suppose that $I=\Sect_{l=1}^m \mm^{\ab_l}$. Then
\[
d(\overline{A(I)})\leq \frac{(n+1)!}{2}d^{n(n-1)},  \quad \text{where} \quad d=\max_{i,l}\{a_{il}\}.
\]
\end{Theorem}

\begin{proof} The proof of the theorem follows the line of arguments in  the proof of Theorem~5.6 in \cite{HHT1}. The inequalities given in Lemma~\ref{belongsto} describe a positive cone $C\subset \RR^{n+1}$. If $q=(c_1,\ldots,c_n, k)$ is an integral vector in $C$, we set $\deg q=k$ and call this number the degree of $q$. Let $E$ be a set of integral vectors spanning the extremal arrays. Then the maximal degree of a generator of $\overline{A(I)}$  is bounded by the maximum of the degree of the integral vectors of the form $a_1q_1+\cdots +a_{n+1}q_{n+1}$ with  $0\leq a_j\leq 1$ and  $a_1 +\cdots +a_{n+1} < n+ 1$, and such that each of the integral vectors $q_j$ spans an extremal array. Hence if $F$ is a set of integral vectors in $C$ which contains the integral vectors spanning the extremal arrays of $C$, we obtain the bound
\begin{eqnarray}
\label{generalbound}
d(\overline{A(I)})\leq (n+1)\max\{\deg q \:\; q\in F\}.
\end{eqnarray}

Note that each integral vector $q$ spanning an extremal array of $C$ is given as an integral solution of $n$ linear equations describing supporting hyperplanes of $C$. Thus there exist numbers  $l_1<l_2<\cdots <l_r$ with $r\leq n$ and numbers $k_1<k_2<\cdots <k_{n-r}$ such that $q$ is an integral solution of the following system  of $n$ homogeneous linear equations
\begin{eqnarray*}
\sum_{i\in \supp(\ab_{l_t})}\prod_{j\in \supp(\ab_{l_t})\atop j\neq i}a_{l_t,j}z_i-\prod_{j\in \supp(\ab_{l_t})}a_{l_t,j}y&=& 0\quad \text{for}\quad t=1,\ldots,r\\
z_{k_s}&=&0\quad \text{for}\quad s=1,\ldots,n-r.
\end{eqnarray*}
which arise from the linear inequalities by clearing denominators.

The coefficient matrix of this linear equation is an $n\times n+1$ matrix $B$. Thus an integral solution of this system is the vector $q$ whose $i$th component is $(-1)^i\det B_i$ where $B_i$ is obtained from $B$ by skipping the $i$th column of $B$. In particular it follows that $\deg q=|\det B_{n+1}|$. Obviously, $\det B_{n+1}$ is equal to a suitable minor of the matrix $A$ whose rows are $\ab_1,\ldots,\ab_m$.

If follows that each such minor is a sum of terms  $\pm g$  where each $g$ is a suitable product of the  $a_{ij}$. If we add up in this sum only  the positive terms, then this   will give an upper bound for this minor. We denote this upper bound of $g$ by $g_+$. Thus it follows from
(\ref{generalbound}) that the desired upper bound for $d(\overline{A(I)})$ is given by $(n+1)f$ where $f$ is bound for the maximal possible $g_+$. The maximal possible $g_+$  we obtain if we consider $n$ minors. Each of these $n$-minors has $n!/2$ positive terms where each of this terms is a
product of $n(n-1)$ entries $a_{l_i}$. Thus $d^{n(n-1)}$ is a bound for each of these entries. Combining all this, we obtain the desired upper bound for $d(\overline{A(I)})$.
\end{proof}


\newpage

\begin{thebibliography}{99}

\bibitem{C} R.C. Cowsik, Symbolic powers and the number of defining equations, Algebra Appl. {\bf 91}, 13–-14, (1985)

\bibitem{R} P. Roberts, A prime ideal in a polynomial ring whose symbolic blow-up is not Noetherian, Proc. Amer. Math.
Soc. {\bf 94}, 589–-592, (1985)

\bibitem{L} G.Lyubeznik, On arithmetical rank of monomial ideals, J. Algebra {\bf 112} , 86--89, (1988)

\bibitem {HH1} J. Herzog, T. Hibi, \textit{Monomial Ideals}, Graduate Texts in Mathematics {\bf 260}, Springer, 2010

\bibitem{HHT1} J. Herzog, T. Hibi, N. V. Trung, Symbolic powers of monomial ideals and vertex cover algebras, Adv. in Math. {\bf 210} ,  304--322, (2007)

\bibitem{HHT2} J. Herzog, T. Hibi, N. V. Trung, Vertex cover algebras  of unimodular hypergraphs, Proc.  Amer. Math. Soc., {\bf 137}, 409–-414, (2009)


\bibitem{HHTZ} J. Herzog, T. Hibi, N.V. Trung and X. Zheng, Standard graded vertex cover algebras, cycles
and leaves, Trans. Amer. Math. Soc. {\bf 360} 6231--6249, (2008)  .


\bibitem {GRV1}I. Gitler, E. Reyes, R. Villarreal, Blowup algebras of ideals of vertex covers of bipartite graph, in: Algebraic
Structures and Their Representations, in: Contemp. Math., vol.{\bf 376}, Amer. Math. Soc.,  273–-279, 2005

\bibitem {SVV} A. Simis, W. Vasconcelos, R. Villarreal, The integral closure of subrings associated to graphs, J. Algebra {\bf  199}, 281-289, (1998)



\end{thebibliography}
\end{document}